\documentclass {article}
\usepackage{amsmath, amssymb, amsthm, amsfonts, amsxtra, latexsym, amscd,
pb-diagram,  graphics,graphicx, color}
\usepackage[bookmarksnumbered, colorlinks, plainpages]{hyperref}
\usepackage [all]{xy}

\newtheorem{thm}{Theorem}
\newtheorem{cor}[thm]{Corollary}
\newtheorem{lem}[thm]{Lemma}
\newtheorem{pro}[thm]{Proposition}
\newtheorem{defn}[thm]{Definition}

\newcommand{\ri}{\rightarrow }

\DeclareMathOperator{\Aut}{Aut} \DeclareMathOperator{\Ext}{Ext}
\DeclareMathOperator{\Ker}{Ker}\DeclareMathOperator{\Coker}{Coker}

\begin{document}

\input diagxy

\centerline{\Large\bf Co-prolongations of a group extension }

\vspace{0.5cm}

\centerline{\bf Nguyen Tien Quang, Doan Trong Tuyen, Nguyen Thi Thu Thuy}

\begin{abstract}  The aim of this paper is to study co-prolongations
of central extensions. We construct the obstruction theory for
co-prolongations
 and classify  the equivalence classes of these by   kernels of a
 homomorphisms
 between  2-dimensional cohomology groups of groups.
\end{abstract}

\vskip 0.2 true cm




\section{\bf Introduction}
\vskip 0.3 true cm

A description of group extensions by means of factor sets leads to a
close relationship between the group extension problem   and the
group cohomology theory \cite{B, M1963, W}. Analogously,  the
extension problems of a type of algebras (such as rings, $\bf
k$-split algebras) were dealt with by  appropriate cohomology
theories (such as Mac Lane cohomology, Hochschild cohomology) (see
\cite{M1958, M1963}).  Quang et al used the group cohomology to
study the prolongation of central extensions in \cite{QPC}. In this
paper we consider the dual of that problem.

Let $G$ be a group and $A$ be a $G$-module.  Consider a short exact
sequences of groups
\begin{align*} 
\begin{diagram}
\xymatrix{ E: 0 \ar[r]& A  \ar[r]^{ \alpha }& B\ar[r]^{\beta }&
G\ar[r] &1,}
\end{diagram}
\end{align*}
 in which the left action by $G=B/A$ on $A$ induced by
$B$-conjugation on the commutative normal subgroup $A$ is the given
$G$-module structure on $A$. We say that $B$ is an {\it extension}
of the group $A$ by the group $G$. A {\it morphism} between two
extensions is a triple of homomorphisms $(\alpha, \beta, \gamma)$
 such that the following diagram commutes

\begin{align}  \begin{diagram}
\xymatrix{ E_0:\;\;\;0 \ar[r]& A_0 \ar[d]^\alpha \ar[r]^{i_0}& B_0 \ar[d]^\beta \ar[r]^{p_0}& G_0 \ar[d]^\gamma \ar[r] &1   \\
E:\;\;\;0 \ar[r]& A \ar[r]^{i}& B \ar[r]^{p}& G \ar[r] &1.}
\end{diagram}\label{eq1}
\end{align}
\indent For a given extension  $E$ and a homomorphism $\gamma$,
there exists an extension  $E_0$ and a morphism $(id_A,\beta,
\gamma)$ (then, $B$ is the fibred product, or the pull-back
$A\times_{G_0}G$) (see \cite{H1970} Ch. IV, \cite{P}). This shows
that $\Ext(G,A)$ is a contravariant functor   in terms of the first
variable, $G$, (\cite{M1963} Ch. III).
Given two extensions $E_0, E$ and two morphisms $\alpha, \gamma$,
Proposition 5.1.1 \cite{W} indicates a necessary and sufficient
condition for there to exist a homomorphism   $\beta$ such that
$(\alpha, \beta, \gamma)$ is a morphism.

Given an extension  $E_0$  and a homomorphism $\gamma:G_0\ri G$, the
problem here is that of finding whether there is any corresponding
extension $E$ such that $E_0=E\gamma$. A particular case when $E_0$
is a central extension and $\gamma:G_0\ri G$ is a normal
monomorphism is studied in \cite{QPC}. In this case $E$ is said to
be a \emph{prolongation} of   the extension $E_0$.
 The
obstruction theory and the classification of prolongations of $E_0$
were dealt with in the case of $\alpha$ is surjective.

The objective of this paper is to solve the problem of prolongations
in the case  $\gamma$ is surjective and $\alpha$ is an identity.
Then, $E$ is termed a \emph{co-prolongation} of   the extension
$E_0$. In Section 2 we show some necessary conditions for the
existence of co-prolongations. We also prove that each such
co-prolongation is just a central extension and it induces a crossed
module. After stating the problem of $\gamma$-{\it co-prolongations}
of a group extension, we construct the obstruction theory for this
concept (Theorem \ref{th1}) in Section 3, and then classify
co-prolongations (Theorem \ref{dlpl}) of the extension $E_0$ by the
kernel of the inducing homomorphism
$$\overline{\gamma}: H^2(G,A)\ri H^2(G_0,A).$$


\section{\bf {\bf \em{\bf Co-prolongations of a group extension}}}
\vskip 0.3 true cm

In this paper  we fix the {\it system} $(E_0, \gamma)$, where $E_0$
is the extension
$$0\rightarrow A\stackrel{i_0}{\rightarrow}B_0\stackrel{p_0}{\rightarrow}G_0\rightarrow 0,$$
and  $\gamma:G_0\ri G$ is injective. Further,  for short, we identify
the abelian group $A$ with the normal subgroup  $i_0(A)$ of $B_0$; 
the operations in $B_0, G_0$ are denoted by the addition, even
though they are are non-necessarily abelian.

Now, suppose that $\alpha=id_A$ in the diagram \eqref{eq1}.
 Then, the extension $E$
is call a {\it co-prolongation} by $\gamma$ (or $\gamma$-{\it
co-prolongation}) of the extension $E_0$.

{\it Remark.} \label{nx1} In the commutative diagram  \eqref{eq1},
since $\alpha=id_A$ and  $\gamma$  is surjective, $\beta$ also is an
{\it surjection}.

\begin{pro}\label{nx2} If the extension $E$ is a
 $\gamma$-{\it co-prolongation}
 of the extension $E_0\in Opext(G_0,A,\varphi_0)$, then there uniquely
 exists a homomorphism
$\varphi : G \rightarrow \Aut A$ such that the following diagram
commutes:
\begin{align}  \begin{diagram}
\xymatrix{ G_0 \ar[dr]_{\varphi_0} \ar[rr]^{\gamma}& & G, \ar[dl]^{\varphi}  &   \\
& \Aut A &}
\end{diagram}
\begin{diagram} \varphi \gamma = \varphi_0. \end{diagram}\label{eq3}
  \end{align}
  \end{pro}
   \begin{proof}
The homomorphism $\varphi : G \rightarrow \Aut A$ induced by the
extension $E$ is determined by
\begin{align}
i[(\varphi pb)a]= b+ ia-b,\;\;b\in B, a\in A.\label{eq2}
 \end{align}
 It is easy to check that  $\varphi$ satisfies \eqref{eq3}.
If $\varphi' : G \rightarrow \Aut(A)$ is a homomorphism satisfying
  \eqref{eq3}, then $\varphi'=\varphi$ since $\gamma$ is surjective. 
\end{proof}

\begin{pro}\label{bode1}  If the extension $E$ is a
 $\gamma$-{\it co-prolongation}
 of the extension $E_0$, there exists an isomorphism
 $j:\Ker \gamma\ri \Ker \beta$
 such that $p_0j=id_{\Ker\gamma}$.
\end{pro}
\begin{proof}
Since the right hand side square of the diagram (\ref{eq1})
commutes,
\begin{equation*} p_0(\Ker \beta)\subset \Ker\gamma. \end{equation*}
Thus, the homomorphism $p_0$ induces a homomorphism
$p^{\bullet}:\Ker\beta \ri \Ker\gamma$.
 We show that $p^{\bullet}$ is surjective. Take $c\in
\Ker\gamma,$ then $c=p_0(x_0)$, where $x_0\in B_0.$ Then,
\begin{equation*} 0=\gamma(c)=\gamma p_0(x_0)=p\beta(x_0). \end{equation*}
It follows that $\beta(x_0)\in \Ker p= A$. Set $a=\beta(x_0)$. Since
the left hand side square commutes,   $\beta(a)=a=\beta(x_0)$, hence
$x_0-a\in \Ker \beta$. One obtains
\begin{equation*} c=p_0(x_0)=p_0(x_0-a)\in p_0(\Ker \beta), \end{equation*}
hence  $p^\bullet$ is surjective. Also, $A\cap \Ker \beta=0$, or
$\Ker(p_0)\cap \Ker \beta=0$, which implies that $p^{\bullet}$ is
injective. Then, $j=(p^{\bullet})^{-1}:\Ker\gamma \ri \Ker\beta$ is
the required isomorphism.
\end{proof}

A monomorphism $j:C\ri D$ is said to be {\it normal} if $jC$ is a
normal subgroup in $D$.

\begin{lem}\label{md1}
If there exists a  normal monomorphism   $j:\Ker\gamma\ri B_0$  such
that $p_0j=id_{\Ker\gamma}$, then
 $p_0^{-1}(\Ker\gamma)=A\times j(\Ker\gamma)$ and the following diagram commutes
\begin{align} \label{cr} \begin{diagram}
\xymatrix{ E': 0 \ar[r]& A \ar@{=}[d] \ar[r]^{i'}& A\times \Ker\gamma
\ar[d]^ \varepsilon \ar[r]^{p'}& \Ker\gamma \ar[d]^\nu \ar[r] &0   \\
 E_0:0 \ar[r]& A \ar[r]^{i}& B_0 \ar[r]^{p}& G_0 \ar[r] &0,}
\end{diagram}
\end{align}
where $A\times \Ker\gamma$ is the direct product, $\nu$ is an
inclusion, $i': a\mapsto (a,0)$,  $p': (a,c)\mapsto c$, and
$\varepsilon : (a,c)\mapsto a+j(c).$
\end{lem}
\begin{proof}
Let $b\in p_0^{-1}(\Ker\gamma)$. Then there exists $c\in \Ker\gamma$
such that $p_0(b)=c=p_0j(c)$. It follows that $b-j(c)=a\in \Ker p_0
= A$. Thus, $b=a+j(c)\in A + j(\Ker\gamma)$. It is easy to see that
$A \cap j(\Ker\gamma)=0$, hence $p_0^{-1}(\Ker\gamma)=A\times
j(\Ker\gamma).$ The map $\alpha: (a,c)\mapsto a+j(c)$ is a
homomorphism and diagram (\ref{cr}) commutes.
\end{proof}

\begin{defn} A {\it crossed module} is a quadruple $\mathcal
M=(B,D,d,\theta)$ in which
  $d:B\ri D,\;\theta:D\ri$ Aut$B$ are group homomorphisms such that\\
\indent $C_1.\ \theta d=\mu$,\\
\indent $C_2.\ d(\theta_x(b))=\mu_x(d(b)),\ x\in D, b\in B$,\\
where $\mu_x$ is is an inner  automorphism given by conjugation with
$x$.
\end{defn}

A crossed module $(B,D,d,\theta)$ is sometimes denoted by
$B\stackrel{d}{\rightarrow}D$, or $B\rightarrow D$.

Crossed modules over groups are introduced by   Whitehead \cite{wh}
(see also \cite{B} Ch. IV, \cite{ML2}). The problem of group
extensions of the type of a crossed module presented in  \cite{BM}.
This closely relates to the problem of prolongations (see
\cite{QPC}). Now, we show that each co-prolongation of a group
extension determines a crossed module.
\begin{pro}
If there exists a  $\gamma$-co-prolongation of $E_0$, then:\\
\indent 1.\ $E_0$ is a central extension,\\
\indent 2.\  $E_0$ induces a homomorphism $\theta: G_0\ri
\Aut(A\times \Ker\gamma)$ such that $(A\times \Ker\gamma, G_0,\nu
p', \theta)$ is a crossed module.
\end{pro}
\begin{proof}
It follows from Proposition  \ref{bode1} and Lemma \ref{md1} that the
diagram  \eqref{cr} commutes.
In this diagram, since $A\subset Z(A\times \Ker\gamma)$, $ E_0$ is a
 $(id_A,\nu)$-{\it prolongation} of the extension $ E'$ in the sense
 of
\cite{QPC}. According to Theorem 10 \cite{QPC}, $ E_0$ is a central
extension.

2) This follows from  Proposition 2 \cite{QPC}. The homomorphism
$\theta:G_0\ri \Aut(A\times \Ker\gamma)$ is given by
$$\theta_g=\phi_{b_0},\ p_0(b_0)=g_0,\qquad \qquad \qquad $$
 $\qquad \qquad \qquad \qquad \phi_{b_0}(x)=\varepsilon^{-1}\mu_{b_0}(\varepsilon x),\ x\in A\times \Ker\gamma.$
 \end{proof}

\section{\bf {\bf \em{\bf The obstruction of a co-prolongation}}}
\vskip 0.3 true cm

In this section, suppose that the isomorphism  $\varphi:G\ri \Aut A$
of the {\it system} $(E_0,\gamma)$ satisfies \eqref{eq3}. The ``co-prolongation problem'' is that of finding whether
  there is any
extension $E$ of $A$ by $G$ which is a co-prolongation of the
extension $E_0$  and, if so, how many of these exist.

Let  $\left\{u(x_0), x_0\in G_0\right\}$, $u(0)=0$, be a set of
representatives of
  $G_0$ in $B_0$.
This set induces a homomorphism $\varphi_0: G_0\rightarrow \Aut A$
by \eqref{eq2} and a factor set  $f_0:(G_0)^2\ri A$ by
$$f_0(x_0,y_0)=u(x_0)+u(y_0)-u(x_0+y_0).$$
According to \cite{M1963} Ch. IV,
  the function $f_0: G_0\times G_0 \rightarrow A$ is uniquely defined
  in terms of
   $B^2_{\varphi_0}(G_0,A)$, that means the element
\begin{equation*} \overline{f_0}=f_0+B^2_{\varphi _0}(G_0,A)
\end{equation*}
 is completely determined. 
Thanks to the relation \eqref{eq3}, the homomorphism $\gamma:
G_0\rightarrow G$ induces one
\begin{align*} \overline{\gamma}: H^2_\varphi (G,A)&\rightarrow H^2_{\varphi _0}(G_0,A)\\
\overline{h}&\mapsto \overline{\gamma^*h},    \end{align*} where
$(\gamma^* h)(x_0,y_0)=h(\gamma x_0,\gamma y_0), \forall x_0,y_0 \in
G_0$.  Then, the element
\begin{equation*}\widetilde{f_0}=\overline{f_0}+\mathrm{Im} (\overline{\gamma} )\in \Coker (\overline{\gamma}),   \end{equation*}
is not dependent on the choice of the representative $u(x_0)$.
 We call $\widetilde{f_0}$ the \emph{obstruction} of $\gamma$-co-prolongation
  of the extension $E_0$.

\vspace{0.2cm}
To prove Theorem \ref{th1} and Corollary \ref{hq1}, one represents
the factor set $f_0$  with respect to the factor set $s:G^2\ri
\Ker\gamma$ of the extension $G_0\stackrel{\gamma}{\rightarrow} G$.
Under the hypothesis of Lemma \ref{md1}, choose a set of
representatives $\left\{u(x_0), x_0\in G_0\right\}$ in $B_0$ as
follows. Firstly, choose a set of representatives
   $\left\{v(x), x\in G\right\},v(1)=0,$ of  $G$
 in $G_0$ whose the corresponding factor set is $s:G^2\ri \Ker\gamma$.
 For each $x\in G$, choose an element  $u_x$ in $B_0$ such that
\begin{equation*} p_0(u_x)=v(x),\;u_1=0. \label{-6}
\end{equation*}
Since the element $x_0\in G_0$ is uniquely written as
\begin{equation*}
 x_0=c+v(x),\;\;c\in \Ker\gamma,\;\;x\in G,
 \end{equation*}
 we set 
\begin{equation} \label{6}
u(x_0)=jc+u_x.
\end{equation}
Let $f_0$ be the factor set of  $B_0$ corresponding to this factor
set. For $x_0,y_0\in G_0$, one has
\begin{equation*} x_0=c+v(x), \; y_0=d+v(y),\;c,d\in \Ker\gamma,\; x,y\in G. \end{equation*}
It follows that
\begin{align*}x_0+y_0&=(c+v(x))+(d+v(y))\\
&=c+\mu_{v(x)}(d)+v(x)+v(y)\\
&= c+\mu_{v(x)}(d)+s(x,y)+v(xy).  \end{align*}
Then,
\begin{equation} \label{c_0}
c_0=c+\mu_{v(x)}(d)+s(x,y)\in \Ker\gamma,
\end{equation}
hence the relation \eqref{6} implies
\begin{equation*}
 u(x_0+y_0)=jc_0+u_{xy}.
 \end{equation*}
  Simple calculations lead to
 \begin{equation}\label{hnt}
 f_0(x_0,y_0)= jc+\mu_{u_x}(jd)+(u_x+u_y-u_{xy})-jc_0.
 \end{equation}

\begin{thm}\label{th1}
Co-prolongations of $E_0$ exist if and only if $\widetilde{f_0}$
vanishes on $\Coker (\overline{\gamma})$.
\end{thm}
\begin{proof}
\emph{Necessary condition.} Let $E$ be a $\gamma$-co-prolongation of $E_0.$ 
Choose in $B_0$ a set of representatives $u(x_0), x_0\in G_0,$ such
that the induced obstruction   $\widetilde{f_0}$ vanishes in
$\Coker(\overline{\gamma} )$.

If $j: \Ker\gamma\rightarrow \Ker \beta$ is the isomorphism
mentioned in   Proposition 
 \ref{bode1}, then $p_0j=id_{\Ker\gamma}$. The set of representatives
  $u(x_0), x_0\in G_0$, chosen by \eqref{6} in $B_0$, gives  a factor
  set $f_0$ satisfying \eqref{hnt}.

Since
\begin{equation*}p\beta (u_x)=\gamma p_0(u_x)=\gamma(v(x))=x,  \end{equation*}
hence  $\{r(x)=\beta(u_x), x\in G\}$ is a set of representatives of
$G$ in $B$
  (clearly, $r(1)=0$).  Let $f$ be a factor set of $B$
  corresponding to this set of representatives,
we prove that
\begin{equation*} f_0=\gamma^* f. \end{equation*}
Since $jc,jc_0,\mu_{u_x}(jd)$ are in $\Ker \beta$, act  $\beta$ on
two sides of the equality  \eqref{hnt} (note that $\beta|_A=id_A$),
one has
\begin{align*}f_0(x_0,y_0)&=\beta u_x+\beta u_y-\beta u_{xy}\\
&=r(x)+r(y)-r(xy)= f(x,y).  \end{align*} It follows that
\begin{equation*}
(\gamma^*f)(x_0,y_0)=f(\gamma x_0,\gamma y_0)=f(x,y)=f_0(x_0,y_0),
 \end{equation*}
that is $\gamma^*f=f_0$, and hence $\widetilde{f}_0$ vanishes in
  $\Coker(\overline{\gamma})$.

 \emph{Sufficient condition.} Let $\widetilde{f_0}=0 \in \text{Coker}(\overline{\gamma})$,
  where $f_0$ is a factor set of $E_0$. There exists $f\in Z^2(G,A)$
  such that
\begin{align*} f_0 = \gamma^\ast(f) + \delta t, \; \delta t \in B^2(G_0,A). \end{align*}
If $\{u(x_0), x_0\in G_0\}$ is a   set of representatives
corresponding to the factor set $f_0$, then one can choose a set of
representatives $u'(x_0) = u(x_0)- t(x)$ so that one obtains a new
factor set
\begin{align*} f'_0(x_0,y_0) = (\gamma^ \ast f)(x_0,y_0). \end{align*}
According to \cite{M1963} Ch. IV, there exists an extension  $E$ of
the crossed product $\overline{B}=[A,\varphi , f, G] $. This is a
$\gamma$-co-prolongation of the extension  $E_0$.
Indeed,  consider the diagram
\begin{align*}  \begin{diagram}
\xymatrix{ E_0:\;\;\;0 \ar[r]& A \ar@{=}[d] \ar[r]^{i_0}& B_0 \ar[d]^\beta \ar[r]^{p_0}& G_0 \ar[d]^\gamma \ar[r] &0   \\
E:\;\;\;0 \ar[r]& A \ar[r]^{i}& \overline{B} \ar[r]^{p}& G \ar[r] &1}
\end{diagram}
\end{align*} where $i: a\mapsto (a,1)$; $p: (a,x)\mapsto x$;  $\beta: a+u'(x_0) \mapsto (a,\gamma x_0).$
Clearly, $\beta$ is a group homomorphism making the above diagram
commute, that means   $E$ is a co-prolongation of  $E_0$.
\end{proof}

\begin{cor}\label{hq1}
Let $(E_0,\gamma)$ and $\varphi:G\ri \Aut A$ satisfy \eqref{eq3}. If
the onto-homomorphism   $\gamma:G_0\ri G$ is split and there is a
 normal monomorphism
   $j:\Ker\gamma\ri B_0$ such that $p_0j=id_{\Ker\gamma}$, then
      $\gamma$-co-prolongations
  of the extension $E_0$ exist.
 \end{cor}
\begin{proof}
Since the onto-homomorphism $\gamma:G_0\ri G$  is split, there is a
normal  monomorphism  $v:G\ri G_0$ such that $\gamma v=id_G$.  Then,
$$G_0=\Ker\gamma \times \mathrm{Im} v.$$
We choose a set of representatives $\left\{v(x)\ |\ x\in G\right\}$
of $G$  in $G_0$ respect to $v$. The corresponding factor set in
$G_0$ is $s=0$. Choose a set of representatives $\left\{u(x_0),
x_0\in G_0\right\}$ by \eqref{6} as in the proof of Theorem
\ref{th1}.
Also,
 by \eqref{6}
\begin{equation*} u(x_0+y_0)=j(c+d)+u_{xy}. \end{equation*}
Since $j$ is a  normal monomorphism,
$$u_x+jd-u_x=j(d'),\ d,d'\in \Ker\gamma.$$
 Act $p_0$ on two sides of the above equality, we have  $v(x)+d-v(x)=d'$.
 Since
  $G_0=\Ker\gamma\times \mathrm{Im }v$, $d'=d$, that means $\mu_{u_x}(jd)=jd$.
Since $s=0$, the relation \eqref{c_0} becomes  $c_0=c+d$. Besides,
$$u_x+u_y-u_{xy}= u(v(x))+u(v(y))-u(v(x)v(y))=f_0(v(x),v(y))\in A.$$
Then, it follows from $A\cap j(\Ker\gamma)=0$ that each element of
  $A$ commutes with each element of
$j(\Ker\gamma)$. Thus, equality \eqref{hnt} turns into
\begin{equation}\label{hnt2}
f_0(x_0,y_0)=f_0(v(x),v(y)).
\end{equation}
Now, define a function    $f:G^2\ri A$ by
\begin{equation}\label{dnf}
f(x,y)=f_0(v(x),v(y)),\ x,y\in G.
\end{equation}
The relation $\eqref{eq3}$ and the fact that $f_0\in
Z^3_{\varphi_0}(G_0,A)$ imply $f\in Z^3_\varphi(G,A)$. Clearly,
\begin{equation*}
(\gamma^*f)(x_0,y_0)=f(\gamma x_0,\gamma y_0)=f((x,y)\stackrel{(\ref{dnf})}{=}
f_0(v(x),v(y))\stackrel{(\ref{hnt2})}{=}f_0(x_0,y_0).
 \end{equation*}
Thus, $\widetilde{f}_0=0$ in $\Coker(\overline{\gamma})$, and hence
by Theorem  \ref{th1}, there exist co-prolongations of $E_0$ by
$\gamma$.
\end{proof}

\begin{thm}[Classification theorem] \label{dlpl} If the system
$(E_0,\gamma)$ together with the homomorphism $\varphi:G\ri \Aut A$
satisfying the relation \eqref{eq3} have  $\gamma$-co-prolongations,
then
 the set of equivalence classes of $\gamma$-co-prolongations
 is a  torseur under the group $K=\Ker(\overline{\gamma})$,
 where
\begin{align*} \overline{\gamma}: H^2_\varphi (G,A)\rightarrow H^2_{\varphi_0}(G_0,A)  \end{align*}
is a homomorphism induced by  $\gamma$.
  \end{thm}
\begin{proof} Firstly, observe that each extension
of  $G$ by $A$ inducing
 $\varphi$ is isomorphic to the extension of the crossed product
 $[A,\varphi , f, G]$,
 where $f$ is uniquely determined up to a coboundary
  $\delta t\in B^2_\varphi (G,A).$

Let $U$ be the set of equivalence classes of co-prolongations of the
extension  $E$. To prove that $U$ is a torseur under $K =
\text{Ker}(\overline{\gamma})$, one constructs a map
\begin{align*}\Lambda:\text{Ker}(\overline{\gamma})=K \rightarrow \Aut (U)   \end{align*}
by the formula
\begin{align*} \Lambda(\overline{h})(cls[A,\varphi, f, G]) = cls[A,\varphi f, h, G]. \end{align*}

Thanks to the above observation, this formula is not dependent on
the representative element of the class  $\overline{h}$, as well as
the representative   $f$. Thus, $\Lambda$ is well defined. Further,
$\Lambda(\overline{h})$ is actually an element of the group of
transformations of $U$. Clearly, $\Lambda$ is a group homomorphism.

It remains to prove that for any two co-prolongations   $E_1, E_2$
of the extension $E$, there uniquely exists  an element
$\overline{h}\in \text{Ker}(\overline{\gamma})$ such that
\begin{align*}cls E_2= \Lambda(\overline{h}) (cls E_1). \end{align*}
Indeed, one has $cls E_1 = cls[G,\varphi, f_i, A], i=1,2,$ where
$\overline{\gamma f_i}=\overline{f}.$ By setting $h= - f_1+ f_2$,
the proof is completed.
\end{proof}

\vskip 0.4 true cm

 {\bf Nguyen Tien Quang}. Department of Mathematics, Hanoi National University of
Education, Vietnam\\
 Email: cn.nguyenquang@gmail.com

 {\bf Doan Trong Tuyen}.  Faculty of  Economics Mathematics, National Economics  University, Hanoi, Vietnam\\
 Email: doantrongtuyen@gmail.com

{\bf Nguyen Thi Thu Thuy}.  School of Applied Mathematics and Informatics, Hanoi University of Science and Technology, Hanoi, Vietnam\\
Email: thuthuyfa@gmail.com
\end{document}